\documentclass[12pt]{article}%
\usepackage{amsmath}
\usepackage{amsfonts}
\usepackage{amssymb}
\usepackage{graphicx}%
\providecommand{\U}[1]{\protect\rule{.1in}{.1in}}
\newtheorem{theorem}{Theorem}

\newtheorem{definition}[theorem]{Definition}

\newtheorem{lemma}[theorem]{Lemma}

\newtheorem{proposition}[theorem]{Proposition}
\newtheorem{remark}[theorem]{Remark}

\newenvironment{proof}[1][Proof]{\noindent\textbf{#1.} }{\ \rule{0.5em}{0.5em}}
\begin{document}

\title{Global bifurcation of traveling waves in discrete nonlinear Schr\"{o}dinger equations}
\author{Carlos Garc\'{\i}a-Azpeitia{\small \thanks{Departamento de Matem\'{a}ticas,
Facultad de Ciencias, Universidad Nacional Aut\'{o}noma de M\'{e}xico, 04510
Ciudad de M\'{e}xico, M\'{e}xico. E-mail: cgazpe@ciencias.unam.mx}}}
\maketitle

\begin{abstract}
The discrete nonlinear Schr\"{o}dinger equations of $n$ sites are studied with
periodic boundary conditions. These equations have $n$ branches of standing
waves that bifurcate from zero. Traveling waves appear as a symmetry-breaking
from the standing waves for different amplitudes. The bifurcation is proved
using the global Rabinowitz alternative in subspaces of symmetric functions.
Applications to the Schr\"{o}dinger and Saturable lattices are presented.



\end{abstract}

\section{Introduction}

The discrete nonlinear Schr\"{o}dinger equation (DNLS) appears in the study of
optical waveguide arrays and Bose--Einstein condensates trapped in optical
lattices \cite{Kr}. In this paper, we consider a general lattice of $n$ sites
described by the equations%
\begin{equation}
i\dot{q}_{j}=V^{\prime}(\left\vert q_{j}\right\vert ^{2})q_{j}+\left(
q_{j+1}-q_{j}\right)  +\left(  q_{j-1}-q_{j}\right)  \text{,} \label{Eq}%
\end{equation}
where the sites, $q_{j}(t)\in\mathbb{C}$ for $j=1,...,n$, satisfy the periodic
boundary conditions $q_{j}=q_{j+n}$. The equations include the Schr\"{o}dinger
lattice, $V(x)=cx^{2}/2$, and the Saturable lattice, $V(x)=c\ln(1+x)$, where
$c$ represents the strength of the linear coupling after rescaling.

Equations (\ref{Eq}) have the explicit solutions
\begin{equation}
q_{j}(t)=ae^{i\omega t}e^{ijm\zeta}\text{ with }\zeta=\frac{2\pi}{n}\text{,}
\label{SW}%
\end{equation}
for $m=1,...,n$, where the frequency $\omega$ is a function of the amplitude
$a\in\mathbb{R}^{+}$ given in (\ref{omega}). These solutions are relative
equilibria and are known as standing waves in the sense that their norms are
stationary in time.

We prove existence of periodic solutions using a non-abelian group that acts
by permutating the oscillators and shifting and reflecting phase and time (see
Definition \ref{Def}). In fact, solutions (\ref{SW}) appear from
symmetry-breaking of the trivial solution; from them, we prove a secondary
symmetry-breaking of periodic solutions (Theorem \ref{Thm1}).

\emph{Main result.} Let $m\in\lbrack0,n/2]\cap\mathbb{N}$ with $m\neq n/4$,
for each $k\in\lbrack1,n/2]\cap\mathbb{N}$ such that
\begin{equation}
\phi_{k}(a)\in(-\infty,1)\backslash\{\gamma_{j}:j\in\lbrack1,n/2]\cap
\mathbb{N}\}\text{,} \label{Con}%
\end{equation}
the relative equilibrium (\ref{SW}) has two global bifurcations of solutions
of the form%
\begin{equation}
q_{j}(t)=e^{i\omega t}e^{ijm\zeta}\left(  a+x\left(  \nu t\pm jk\zeta\right)
\right)  \text{,} \label{PS}%
\end{equation}
where
\[
\phi_{k}(a)=\frac{a^{2}V^{\prime\prime}(a^{2})}{2\cos m\zeta\sin^{2}%
\frac{k\zeta}{2}}\qquad\text{and }\qquad\gamma_{k}=1-\cot^{2}\frac{k\zeta}%
{2}\tan^{2}m\zeta\text{.}%
\]
Each branch is a global continuum in the space of $2\pi$-periodic functions
$x$ and frequencies $\nu$ emanating from $(0,\nu_{k}^{\pm})$ given in
(\ref{nuk}).

These solutions are traveling waves in the sense that their norms satisfy
\[
\left\vert q_{j}\right\vert (t)=a+r(\nu t\pm jk\zeta)\text{,}%
\]
where $r(t)$ is real $2\pi$-periodic; these solutions are known as traveling
or moving breathers when they are localized. The existence of localized
traveling waves in infinite lattices is proved in \cite{PeRo} (see also
Chapter 16 and 21 in \cite{Kr} and \cite{Ro10}).

In Theorem \ref{Thm2} we prove that solutions (\ref{SW}) are stable if the
conditions (\ref{Con}) hold for $k=1,...,n-1$. In \cite{Kr}, solutions given
by (\ref{SW}) are called plane waves, the nonlinear dispersion relation\ (6.7)
is equivalent to (\ref{omega}) and the modulation stability (6.8) has stable
directions precisely when (\ref{Con}) holds.

In \cite{GaIz11}, the authors find a bifurcation of relative equilibria for
$m=1$ and amplitudes $\phi_{k}(a)=\gamma_{k}$. This bifurcation exists due to
an eigenvalue of the linearization that crosses zero. This phenomena occurs
for any $m$, and we should expect a bifurcation of standing waves appearing
from the amplitudes $\phi_{k}(a)=\gamma_{k}$.

For $\phi_{k}(a)\geq1$, there are two eigenvalues colliding on the imaginary
axis and detaching into the complex plane. This phenomenon may trigger a
Hamiltonian-Hopf bifurcation where isolated traveling waves persist for
$\phi_{k}(a)>1$. Indeed, a bifurcation of this kind is described in
\cite{Jo04} for the trimer $n=3$.

The authors prove in \cite{GaIz13,GaIz12,GaIz13-2} bifurcation of periodic
solutions for bodies, vortices and Schr\"{o}dinger sites, for $m=1$. In the
body and vortex problems, the coupling is homogeneous and invariant under all
permutations, i.e. the stability and bifurcation properties of (\ref{SW}) are
independent of $a$ and $m$; but this is not the case for Schr\"{o}dinger
sites. We complete the analysis of bifurcation and stability for all $m$'s.

References \cite{GaIz13,GaIz12,GaIz13-2} use a global Lyapunov-Schmidt
reduction and a topological degree for $G$-equivariant maps that are
orthogonal to the generators (see \cite{BaKr10,DaGe05,IzVi03}). In this paper
we present a direct and self-contained approach, requiring non-abelian group
actions and the global Rabinowitz alternative \cite{Ra}.

In Section 2, we define the equivariant properties of the bifurcation
operator, we present a reduction to a finite number of Fourier components, and
we find the spectra. In Section 3.1 we analyze the spectra for each
irreducible representation and we prove the bifurcation result. In Section 3.2
we study the stability. In Section 4, we apply the results to the focusing and
defocusing Schr\"{o}dinger and Saturable lattices.

\section{Setting the problem}

Equations (\ref{Eq}) in rotating coordinates, $q_{j}(t)=e^{i\omega t}u_{j}%
(t)$, are%
\begin{equation}
i\dot{u}_{j}-\omega u_{j}=V^{\prime}(\left\vert u_{j}\right\vert ^{2}%
)u_{j}+\left(  u_{j+1}-u_{j}\right)  +\left(  u_{j-1}-u_{j}\right)  \text{.}
\label{ec}%
\end{equation}
The values $a_{j}=ae^{ijm\zeta}$ satisfy $a_{j}=a_{j+n}$ and%
\[
\left(  a_{j+1}-2a_{j}+a_{j-1}\right)  =(-4\sin^{2}m\zeta/2)a_{j}\text{.}%
\]
Then $u_{j}(t)=a_{j}$ is an equilibrium and (\ref{SW}) is a solution of
(\ref{Eq}), when%
\begin{equation}
\omega=4\sin^{2}m\zeta/2-V^{\prime}(a^{2})\text{.} \label{omega}%
\end{equation}

In real coordinates, $u_{j}\in\mathbb{R}^{2}$, equations (\ref{ec}) are
\[
J\dot{u}_{j}=\omega u_{j}+V^{\prime}(\left\vert u_{j}\right\vert ^{2}%
)u_{j}+(u_{j+1}-2u_{j}+u_{j-1})\text{,}%
\]
where $J$ is the symplectic matrix%
\begin{equation}
J=\left(
\begin{array}
[c]{cc}%
0 & -1\\
1 & 0
\end{array}
\right)  \text{ and }R=\left(
\begin{array}
[c]{cc}%
1 & 0\\
0 & -1
\end{array}
\right)  \text{.} \label{R}%
\end{equation}

Let $u=(u_{1},...,u_{n})$ and $\mathcal{J}=diag(J,...,J)$. The system of
equations in vectorial form is
\[
\mathcal{J}\dot{u}=\nabla H(u),
\]
where $H$ is the Hamiltonian
\begin{equation}
H=\frac{1}{2}\sum_{j=1}^{n}\left\{  V(\left\vert u_{j}\right\vert ^{2}%
)+\omega\left\vert u_{j}\right\vert ^{2}-\left\vert u_{j+1}-u_{j}\right\vert
^{2}\right\}  \text{.} \label{Heq}%
\end{equation}
In this setting, the relative equilibrium is
\begin{equation}
\mathbf{a}_{m}=(a_{1},...,a_{n})\text{ with }a_{j}=ae^{jm\zeta J}e_{1}\text{.}
\label{Equil}%
\end{equation}

Using the change of variables $u(t)=\mathbf{a}_{m}+x(\nu t)$, $2\pi/\nu
$-periodic solutions of the Hamiltonian system correspond to zeros of the
operator
\[
f(x;\nu)=\mathcal{J}\dot{x}-\nu^{-1}\nabla H(\mathbf{a}_{m}+x):H_{2\pi}%
^{1}(\mathbb{R}^{2n})\times\mathbb{R}^{+}\rightarrow L_{2\pi}^{2}%
(\mathbb{R}^{2n})\text{.}%
\]
Below we will prove global bifurcation of periodic solutions from the set of
trivial solutions $(0,\nu)$ for $\nu\in\mathbb{R}^{+}$.

\begin{definition}
\label{Def}Let $m$ be as in (\ref{Equil}). The linear action of the
non-abelian group%
\[
G=(\mathbb{Z}_{n}\times S^{1})\cup\kappa(\mathbb{Z}_{n}\times S^{1})
\]
in $L_{2\pi}^{2}(\mathbb{R}^{2n})$ is given by the homomorphism of groups
$\rho:G\rightarrow GL(L_{2\pi}^{2})$ generated by
\begin{align*}
\rho(\zeta,\varphi)x_{j}(t)  &  =e^{-m\zeta J}x_{j+1}(t+\varphi)\text{,}\\
\rho(\kappa)x_{j}(t)  &  =Rx_{n-j}(-t)\text{,}%
\end{align*}
where $\zeta=2\pi/n$ generates $\mathbb{Z}_{n}$ in $S^{1}=[0,2\pi)$ and
$\kappa$ is the reflection that generates two copies of $\mathbb{Z}_{n}\times
S^{1}$ in $G$.
\end{definition}

Since the Hamiltonian $H$ is gauge invariant and autonomous, the map $f$ is
$G$-equivariant. Moreover, all the elements of $G$ leave the equilibrium
$\mathbf{a}_{m}$ fixed, and then the isotropy group of $\mathbf{a}_{m}$ is
$G_{\mathbf{a}_{m}}=G$.

\subsection{Irreducible representations}

In this section we find the irreducible representations of the action of $G$.
We define the isomorphisms $T_{k}:\mathbb{C}^{2}\rightarrow V_{k}$ as%
\begin{equation}
T_{k}z=n^{-1/2}(e^{(ikI+mJ)\zeta}z,...,e^{n(ikI+mJ)\zeta}z)\in\mathbb{C}%
^{2n}\text{,}%
\end{equation}
where $V_{k}$ is the image of $T_{k}$ for $k\in\{1,...,n\}$.

A function $x(t)\in L_{2\pi}^{2}(\mathbb{R}^{2n})$ is expressed in the Fourier
components as $x=\sum_{l\in\mathbb{Z}}x_{l}e^{ilt}$, and each Fourier
component $x_{l}\in\mathbb{C}^{2n}$ as $x_{l}=\sum_{k=1}^{n}T_{k}x_{k,l}$ with
$x_{k,l}\in\mathbb{C}^{2}$. Therefore, a function can be expressed in the
orthonormal coordinates $T_{k}x_{k,l}e^{ilt}$ as
\begin{equation}
x(t)=\sum_{(k,l)\in\mathbb{Z}_{n}\times\mathbb{Z}}T_{k}x_{k,l}e^{ilt}%
\text{,\qquad}\mathbb{Z}_{n}=\{1,...,n\}\text{.}%
\end{equation}

Let us denote the $j$-th component of $T_{k}x_{k,l}\in\mathbb{C}^{2n}$ by the
two-dimensional vector $n^{-1/2}e^{j(ikI+mJ)\zeta}x_{k,l}\in\mathbb{C}^{2}$.
With this notation, the $j$-th component of $\rho(\zeta)T_{k}x_{k,l}$ is%
\[
n^{-1/2}e^{-mJ\zeta}e^{(j+1)(ikI+mJ)\zeta}x_{k,l}=n^{-1/2}e^{j(ikI+mJ)\zeta
}(e^{ik\zeta}x_{k,l})\text{,}%
\]
then
\[
\rho(\zeta)T_{k}x_{k,l}=e^{ik\zeta}T_{k}x_{k,l}\text{.}%
\]
Similarly,
\[
\rho(\kappa)T_{k}x_{k,l}=T_{k}R\bar{x}_{k,l}\text{.}%
\]

Therefore, the subspaces of similar irreducible representations of the space
$L_{2\pi}^{2}(\mathbb{R}^{2n})$ by the action of $G$ are
\[
V_{k,l}=\{T_{k}x_{k,l}e^{ilt}:x_{k,l}\in\mathbb{C}^{2}\}.
\]
In the components $x_{k,l}$, the action of $G$ is
\begin{equation}
\rho(\zeta)x_{k,l}=e^{ik\zeta}x_{k,l}\text{,\qquad}\rho(\varphi)x_{k,l}%
=e^{il\varphi}x_{k,l}\text{,\qquad}\rho(\kappa)x_{k,l}=R\bar{x}_{k,l}\text{.}%
\end{equation}

\subsection{Linearization}

Since $V_{k,l}$ are the subspaces of similar irreducible representations of
$G$ and
\[
f^{\prime}(0)=\mathcal{J}\partial_{t}-\nu^{-1}D^{2}H(\mathbf{a}_{m})
\]
is $G$-equivariant, by Schur's lemma, the linearization $f^{\prime}(0)$ is
block diagonal in the components $V_{k,l}$. The diagonal decomposition can be
obtained explicitly using the following proposition,
\[
f^{\prime}(0)x=\sum_{(k,l)\in\mathbb{Z}_{n}\times\mathbb{Z}}T_{k}%
(ilJ-B_{k})x_{k,l}e^{ilt}.
\]

\begin{proposition}
Let $\alpha_{k}$\ and $\beta_{k}$\ be%
\begin{equation}
\alpha_{k}=4\cos m\zeta\sin^{2}k\zeta/2\text{,\qquad}\beta_{k}=2\sin
m\zeta\sin k\zeta\text{,}%
\end{equation}
the Hessian $D^{2}H(\mathbf{a}_{m})$ satisfy that%
\[
D^{2}H(\mathbf{a}_{m})T_{k}z=T_{k}B_{k}z\text{,}%
\]
where $B_{k}$ is the $2\times2$ matrix%
\begin{equation}
B_{k}=\emph{diag}(2a^{2}V^{\prime\prime}(a^{2})-\alpha_{k},-\alpha
_{k})+iJ\beta_{k}\text{.}%
\end{equation}

\end{proposition}

\begin{proof}
We express $D^{2}H(\mathbf{a}_{m})$ in $2\times2$ blocks $A_{i,j}$ as
\[
D^{2}H(\mathbf{a}_{m})=(A_{i,j})_{i,j=1}^{n}\text{.}%
\]
Since the coupling in the DNLS equations happens only between adjacent sites,
then $A_{i,j}=I$ for $\left\vert i-j\right\vert =1$ and $A_{i,j}=0$ for
$\left\vert i-j\right\vert >1$, modulus $n$.

Using $a_{j}=ae^{jm\zeta J}e_{1}$, we have
\[
\frac{1}{2}D^{2}V(\mathbf{a}_{m})=\ V^{\prime}(a^{2})I+2a^{2}V^{\prime\prime
}(a^{2})e^{(jm\zeta)J}e_{1}e_{1}^{T}e^{-(jm\zeta)J},
\]
where $e_{1}e_{1}^{T}=diag(1,0)$. Since $V^{\prime}(a^{2})+\omega-2=-2\cos
m\zeta$, we conclude%
\[
A_{j,j}=-2(\cos m\zeta)I+2a^{2}V^{\prime\prime}(a^{2})e^{(jm\zeta)J}e_{1}%
e_{1}^{T}e^{-(jm\zeta)J}\text{.}%
\]

Given that the $j$-th component of $T_{k}z\in\mathbb{C}^{2n}$ is
$n^{-1/2}e^{j(ikI+mJ)\zeta}z\in\mathbb{C}^{2}$, where $z\in\mathbb{C}^{2}$,
the $j$-th component of $D^{2}H(a)T_{k}z$ is%
\[
\frac{1}{\sqrt{n}}\left(  A_{j,j}+e^{(ikI+mJ)\zeta}+e^{-(ikI+mJ)\zeta
})\right)  e^{j(ikI+mJ)\zeta}z=\frac{1}{\sqrt{n}}e^{j(ikI+mJ)\zeta}\left(
B_{k}z\right)  \text{,}%
\]
where%
\[
B_{k}=-2(\cos m\zeta)I+2a^{2}V^{\prime\prime}(a^{2})e_{1}e_{1}^{T}%
+e^{(ikI+mJ)\zeta}+e^{-(ikI+mJ)\zeta}\text{.}%
\]

From the equalities
\[
e^{(ikI+mJ)\zeta}+e^{-(ikI+mJ)\zeta}=(2\cos k\zeta\cos m\zeta)I+(2\sin
k\zeta\sin m\zeta)iJ\text{,}%
\]
and
\[
-2\cos m\zeta(1-\cos k\zeta)=-4\cos m\zeta(\sin k\zeta/2)^{2}=-\alpha
_{k}\text{,}%
\]
we conclude that%
\[
B_{k}=-\alpha_{k}I+\beta_{k}(iJ)+2a^{2}V^{\prime\prime}(a^{2})diag(1,0).
\]

\end{proof}

\section{Main Results}

The symmetries permit us to assume, without loss of generality, that
$m\in\lbrack0,n/2]\cap\mathbb{N}$. Throughout this section we also assume
$m\neq n/4$.

In these cases the sign of $\alpha_{k}$ is well defined for $k=1,...,n-1$,%
\[
sgn(\alpha_{k})=\left\{
\begin{array}
[c]{c}%
1\text{ if }m\in\lbrack0,n/4)\text{ }\\
-1\text{ if }m\in(n/4,n/2]
\end{array}
\right.  \text{,}%
\]
and we can define%
\begin{equation}
\phi_{k}(a)=\frac{2a^{2}}{\alpha_{k}}V^{\prime\prime}(a^{2})\text{,\qquad
}\gamma_{k}=1-\left(  \frac{\beta_{k}}{\alpha_{k}}\right)  ^{2}\text{.}%
\end{equation}
Therefore, the matrix $D^{2}H(\mathbf{a}_{m})\ $is block diagonal with blocks
\[
B_{n}=\emph{diag}(2a^{2}V^{\prime\prime}(a^{2}),0)\text{,\qquad}B_{k}%
=\alpha_{k}diag(\phi_{k}-1,-1)+\beta_{k}(iJ)\text{,}%
\]
for $k=1,..,n-1$.

\subsection{Bifurcation theorem}

We consider bifurcation in the fixed point space of the isotropy group
$\tilde{D}_{n}$ generated by $\left(  \zeta,-k\zeta\right)  $ and $\kappa$.
Solutions in the fixed point space of $\tilde{D}_{n}$ have symmetries%
\begin{equation}
x_{j}(t)=\rho(\zeta,-k\zeta)x_{j}(t)=e^{-m\zeta J}x_{j+1}(t-k\zeta)\text{.}
\label{sym}%
\end{equation}
That is $x_{j}(t)=e^{jm\zeta J}x_{n}(t+jk\zeta)$ and, by the action of
$\kappa$,%
\[
x_{n}(t)=\rho(\kappa)x_{n}=Rx_{n}(-t)\text{.}%
\]

The component $x_{k,1}\in\mathbb{C}^{2}$ is fixed by $\kappa\in\tilde{D}_{n}$
if $x_{k,1}\in\mathbb{R}\times i\mathbb{R}$.

\begin{lemma}
\label{Prop}The matrix $iJB_{k}$ in the subspace $\mathbb{R}\times
i\mathbb{R}$ has eigenvalues
\begin{equation}
\nu_{k}^{\pm}=\beta_{k}\pm\sqrt{\alpha_{k}^{2}(1-\phi_{k})}\text{.}
\label{nuk}%
\end{equation}
Moreover,

\begin{description}
\item[(a)] If $k\in\lbrack1,n-1]\cap\mathbb{N}$ and $\phi_{k}(a)\in
(-\infty,\gamma_{k})$, then $\nu_{k}^{+}$ is positive.

\item[(b)] If $k\in\lbrack1,n/2)\cap\mathbb{N}$ and $\phi_{k}(a)\in(\gamma
_{k},1)$, then $\nu_{k}^{+}$ and $\nu_{k}^{-}$ are positive.
\end{description}
\end{lemma}

\begin{proof}
Let $L=\mathrm{diag}(1,i)$. The eigenvalues of $iJB_{k}$ in the subspace
$\mathbb{R}\times i\mathbb{R}$ are the eigenvalues of the real matrix%
\[
L^{-1}(iJB_{k})L=\left(
\begin{array}
[c]{cc}%
\beta_{k} & -\alpha_{k}\\
\alpha_{k}\left(  \phi_{k}-1\right)  & \beta_{k}%
\end{array}
\right)  \text{.}%
\]
The eigenvalues $\nu$ of this matrix are the zeros of
\begin{equation}
d_{k}(\nu)=(\nu-\beta_{k})^{2}-\alpha_{k}^{2}(1-\phi_{k})\text{.} \label{det}%
\end{equation}

The function $d_{k}(0)$ has two real solutions if and only if $\phi_{k}%
\in(-\infty,1)$. Since $d_{k}(\nu)$ is a polynomial of order $\nu^{2}$ at
infinity, then $\nu_{k}^{+}$ is positive and $\nu_{k}^{-}$ is negative if
$d_{k}(0)$ is negative. This is the case if $\phi_{k}<1-(\beta_{k}/\alpha
_{k})^{2}=\gamma_{k}$. For $\phi\in(\gamma_{k},1)$, the function $d_{k}(0)$ is
positive, and the values $\nu_{k}^{\pm}$ have the same sign of $\beta_{k}$. We
conclude this result from the fact that $\beta_{k}>0$ for $k\in\lbrack
1,n/2)\cap\mathbb{N}$.
\end{proof}

\begin{definition}
We say that the amplitude $a$ is non-degenerate if $V^{\prime\prime}%
(a^{2})\neq0$, $\phi_{k}(a)\neq\gamma_{k}$ for $k=1,...,n-1$. We say that the
frequency $\nu_{k}^{\pm}$ for $k=1,...,n-1$ is non-resonant if $\nu_{j}^{\pm
}\neq l\nu_{k}^{\pm}$ for integers $l\geq1$ and $j\neq k$.
\end{definition}

The matrix $B_{n}$ restricted to the subspace $x_{n,0}=\rho(\kappa
)x_{n,0}=Rx_{n,0}$ has the simple eigenvalue $2a^{2}V^{\prime\prime}(a^{2})$.
For $k=1,...,n$, the blocks $B_{k}$ have determinants $\beta_{k}^{2}%
-\alpha_{k}^{2}(1-\phi_{k})$. Therefore, the non-degeneracy property of $a$
assures that $D^{2}H(\mathbf{a}_{m})$ has no zero-eigenvalues in
$\mathrm{Fix}(\tilde{D}_{n})$.

\begin{theorem}
\label{Thm1}Assume that $a$ is non-degenerate. If $\nu_{k}^{\pm}$ is
non-resonant, then

\begin{description}
\item[(a)] For each $k\in\lbrack1,n-1]\cap\mathbb{N}$ such that $\phi
_{k}(a)\in(-\infty,\gamma_{k})$, there is a global bifurcation from
$(0,\nu_{k}^{+})$ in the space
\begin{equation}
\{x\in H_{2\pi}^{2}:x_{j}(t)=e^{jm\zeta J}x_{n}(t+jk\zeta),x_{n}%
(t)=Rx_{n}(-t)\}\times\mathbb{R}^{+} \label{space}%
\end{equation}

\item[(b)] For $k\in\lbrack1,n/2]\cap\mathbb{N}$ such that $\phi_{k}%
(a)\in(\gamma_{k},1)$, there are global bifurcations from $(0,\nu_{k}^{+})$
and $(0,\nu_{k}^{-})$ in the space given by (\ref{space}).
\end{description}
\end{theorem}

\begin{proof}
Let $K:L_{2\pi}^{2}\rightarrow H_{2\pi}^{1}$ be the operator defined in the
Fourier basis $x=\sum_{l\in\mathbb{Z}}x_{l}e^{ilt}$ as%
\[
Kx=x_{0}+\sum_{l\in\mathbb{Z}\backslash\{0\}}(li\mathcal{J})^{-1}x_{l}%
e^{ilt}\text{.}%
\]
Since $K:H_{2\pi}^{1}\rightarrow H_{2\pi}^{1}$ is compact and%
\begin{equation}
f(x)=\mathcal{J}\dot{x}-\nu^{-1}D^{2}H(\mathbf{a}_{m})x+\mathcal{O(}\left\vert
x\right\vert ^{2})\text{,}%
\end{equation}
then%
\begin{equation}
Kf(x)=x-T(\nu)x+g(x):H_{2\pi}^{1}\rightarrow H_{2\pi}^{1}\text{,}%
\end{equation}
where $T(\nu)$ is the compact linear operator%
\[
Tx=(I+\nu^{-1}D^{2}H(\mathbf{a}_{m}))x_{0}+\sum_{l\in\mathbb{Z}\backslash
\{0\}}(\nu l)^{-1}i\mathcal{J}D^{2}H(\mathbf{a}_{m})x_{l}e^{ilt}\text{,}%
\]
and $g(x)=\mathcal{O(}\left\vert x\right\vert _{H_{2\pi}^{1}}^{2})$ is a
nonlinear compact operator.

Since $f$ and $K$ are $G$-equivariant, the operator $Kf$ is $G$-equivariant.
Then $Kf$ is well defined in the space $\mathrm{Fix}(\tilde{D}_{n})$ given in
(\ref{space}). The global bifurcation follows from Theorem 3.4.1 in
\cite{Ni2001} if $T(\nu)$ has a simple eigenvalue crossing $1$ in
$\mathrm{Fix}(\tilde{D}_{n})$.

For $l=0$, by the non-degeneracy property of $a$, the matrix $I+\nu^{-1}%
D^{2}H(\mathbf{a}_{m})$ has no eigenvalues equal to $1$ in $\mathrm{Fix}%
(\tilde{D}_{n})$. For $l\geq2$, due to the non-resonance property of $\nu
_{k}^{\pm}$, the matrix $(l\nu_{k}^{\pm})^{-1}i\mathcal{J}D^{2}H(\mathbf{a}%
_{m})$ has no eigenvalues equal to $1$. For $l=1$, we have that $\nu
^{-1}i\mathcal{J}D^{2}H(\mathbf{a}_{m})$ has an eigenvalue crossing $1$ when
$\nu$ crosses $\nu_{k}^{\pm}$, corresponding to the block $iJB_{k}$. Moreover,
this eigenvalues is simple when $iJB_{k}$ is restricted to $\mathbb{R}\times
i\mathbb{R}$. We conclude that $T(\nu)$ in $\mathrm{Fix}(\tilde{D}_{n})$ has a
simple eigenvalue crossing $1$ when $\nu$ crosses $\nu_{k}^{\pm}$.
\end{proof}

\begin{remark}
In the case of $1:l$ resonances, $l\nu_{k}^{\pm}=\nu_{j}^{\pm}$ for $l\geq2$,
the previous theorem gives the existence of the bifurcation with the biggest
frequency $\nu_{j}^{\pm}$. In the case of $1:1$ resonances, $\nu_{k}^{+}%
=\nu_{k}^{-}$, we cannot prove existence of bifurcation because there is a
double eigenvalue of $T(\nu)$ crossing $1$. In this case, a Hamiltonian-Hopf
bifurcation may appear at $\phi_{k}(a)=1$. This is described in Theorem 11.5.1
of \cite{MeHa91}, where two isolated solutions persist for $\phi_{k}(a)>1$.
\end{remark}

\subsection{Stability Analysis}

Let%
\begin{equation}
\sigma_{m}=\left\{
\begin{array}
[c]{c}%
sgn(V^{\prime\prime}(a^{2}))\text{ if }m\in\lbrack1,n/4)\\
-sgn(V^{\prime\prime}(a^{2}))\text{ if }m\in(n/4,n/2]
\end{array}
\right.  \text{.}%
\end{equation}

\begin{theorem}
\label{Thm2}If $\sigma_{m}<0$, or $\sigma_{m}>0$ and $\phi_{1}(a)<1$, then the
relative equilibrium (\ref{SW}) is linearly stable.
\end{theorem}

\begin{proof}
The Hamiltonian equation (\ref{Heq}) is linearly stable at the equilibrium
$\mathbf{a}_{m}$ if $\mathcal{J}D^{2}H(\mathbf{a}_{m})$ has only pairs of
purely conjugated imaginary eigenvalues $\pm i\nu$. Since $\mathcal{J}%
D^{2}H(\mathbf{a}_{m})$ has a pair of zero eigenvalues that corresponds to the
gauge symmetry, the system is linearly stable if $\mathcal{J}D^{2}%
H(\mathbf{a}_{m})$ has $n-1$ pairs of purely imaginary eigenvalues.

The sign of $\phi_{k}$ does not depend on $k\in\{1,...,n-1\}$ and is equal to
$\sigma_{m}$. For $\sigma_{m}$ negative, we have $\phi_{k}(a)<1$. For
$\sigma_{m}$ positive, using the fact that $\cos k\zeta$ is increasing for
$k\in\lbrack1,n/2]\cap\mathbb{N}$, then $\phi_{k}$ is decreasing for
$k=1,...,n/2$ and $\phi_{k}(a)<\phi_{1}(a)<1$. In both cases, it is not
difficult to see that $\mathcal{J}D^{2}H(\mathbf{a}_{m})$ has eigenvalues
$i\nu_{k}^{\pm}$ for $k=1,...,n-1$. Since $\nu_{n-k}^{\pm}=-\nu_{k}^{\pm}$, we
conclude that $\mathcal{J}D^{2}H(\mathbf{a}_{m})$ has $n-1$ pairs of purely
imaginary eigenvalues.
\end{proof}

\section{Applications}

In complex coordinates,%
\[
q_{j}(t)=e^{i\omega t}e^{ijm\zeta}\left(  a+x_{n}(\nu t+jk\zeta)\right)
\text{.}%
\]
These solutions are discrete traveling waves in the sense that the norms
satisfy%
\[
\left\vert q_{j}\right\vert (t)=a+\left\vert x_{n}\right\vert (\nu
t+jk\zeta).
\]
For example, if $k$ divides $n$, the traveling wave has $k$ identical waves
with wavelength equal to $n/k$ sites.

In the following sections, we present applications to the Schr\"{o}dinger and
Saturable lattice.

\subsection{Schr\"{o}dinger lattice}

The cubic Schr\"{o}dinger potential is $V(x)=cx^{2}/2$, where $c>0$
corresponds to the focusing case, and $c<0$ to the defocusing case. Standing
waves given by (\ref{SW}) exist for
\[
\omega=4\sin^{2}m\zeta/2-ca^{2}\text{.}%
\]

In the focusing case, the potential satisfies%
\[
V^{\prime\prime}(a^{2})=c>0\text{ and }\sigma=1\text{.}%
\]
If $m\in\lbrack0,n/4)$, then $\alpha_{k}>0$ and $\phi_{k}(a)<1$ for
$a^{2}<\alpha_{k}/2c$. If $m\in(n/4,n/2]$, then $\alpha_{k}<0$ and $\phi
_{k}(a)<1$ for all $a$.

Therefore, the following result holds.

\begin{proposition}
In the focusing Schr\"{o}dinger lattice ($c>0$), the equilibrium
$\mathbf{a}_{m}$ is linearly stable when $m\in(n/4,n/2]$, or $m\in
\lbrack0,n/4)$ and $a<\sqrt{\alpha_{1}/2c}$. Moreover, the equilibrium
$\mathbf{a}_{m}$ has two global bifurcations of traveling waves\ for each
$k\in\lbrack1,n/2]\cap\mathbb{N}$ if $m\in(n/4,n/2]$, or $m\in\lbrack0,n/4)$
and $a<\sqrt{\alpha_{k}/2c}$.
\end{proposition}

In the defocusing case, the potential satisfies
\[
V^{\prime\prime}(a^{2})=c<0\text{ and }\sigma=-1\text{.}%
\]
For $m\in\lbrack0,n/4)$, $\phi_{k}(a)<1$ . For $m\in(n/4,n/2]$, if
$a^{2}<\alpha_{k}/2c$, then $\phi_{k}(a)<1$ .

\begin{proposition}
In the defocusing Schr\"{o}dinger lattice ($c<0$), the equilibrium
$\mathbf{a}_{m}$ is linearly stable when $m\in\lbrack0,n/4)$, or
$m\in(n/4,n/2]$ and $a<\sqrt{\alpha_{1}/2c}$. Moreover, the equilibrium
$\mathbf{a}_{m}$ has two global bifurcations of traveling waves for each
$k\in\lbrack1,n/2]\cap\mathbb{N}$ if $m\in\lbrack0,n/4)$, or $m\in(n/4,n/2]$
and $a<\sqrt{\alpha_{k}/2c}$.
\end{proposition}

\subsection{Saturable lattice}

The Saturable potential is given by $V(x)=c\ln(1+x)$ with $c>0$. It is clear
that%
\[
V^{\prime\prime}(a^{2})=-c(1+a^{2})^{-2}\text{ and }\sigma=-1\text{.}%
\]
If $m\in\lbrack0,n/4)$, then $\phi_{k}(a)<1$ for all $a$, while if
$m\in(n/4,n/2]$, then $\phi_{k}(a)<1$ for $\left(  a+a^{-1}\right)
^{-2}<-\alpha_{k}/2c$.

\begin{proposition}
The equilibrium $\mathbf{a}_{m}$ in the Saturable lattice is linearly stable
when $m\in\lbrack0,n/4)$, or $m\in(n/4,n/2]$ and $\left(  a+a^{-1}\right)
^{-1}<\sqrt{-\alpha_{1}/2c}$. Moreover, the equilibrium $\mathbf{a}_{m}$ has
two global bifurcations of traveling waves for each $k\in\lbrack
1,n/2]\cap\mathbb{N}$ if $m\in\lbrack0,n/4)$, or $m\in(n/4,n/2]$ and%
\[
\left(  a+a^{-1}\right)  ^{-1}<\sqrt{-\alpha_{k}/2c}\text{.}%
\]

\end{proposition}

\textbf{Acknowledgements.} C. Garc\'{\i}a is grateful to P.~Panayotaros, M.
Tejada-Wriedt and the referee and editor for their useful comments which
greatly improved the presentation of this manuscript.


\begin{thebibliography}{99}                                                                                               %


\bibitem {BaKr10}Z. Balanov, W. Krawcewicz, S. Rybicki, and H. Steinlein,
\emph{A short treatise on the equivariant degree theory and its applications,}
J. Fixed Point Theory Appl. 8 (2010), pp. 1--74.

\bibitem {DaGe05}E. Dancer, K. Geba, and S. Rybicki, \emph{Classification of
homotopy classes of equivariant gradient maps,} Fund. Math. 185 (2005), pp. 1-18.

\bibitem {Ro10}M. Feckan and V. Rothos, \emph{Travelling waves of discrete
nonlinear Schrodinger equations with nonlocal interactions, }Applicable
Analysis 89 (2010), pp. 1387-1411.

\bibitem {GaIz13}C.~Garc\'{\i}a-Azpeitia and J.~Ize, \emph{Global bifurcation
of planar and spatial periodic solutions from the polygonal relative
equilibria for the }$n$\emph{-body problem,} J. Differential Equations 254
(2013), pp. 2033--2075.

\bibitem {GaIz13-2}C.~Garc\'{\i}a-Azpeitia and J.~Ize, \emph{Bifurcation of
periodic solutions from a ring configuration of discrete nonlinear
oscillators,} DCDS-S 6 (2013), pp. 975 - 983.

\bibitem {GaIz12}C.~Garc\'{\i}a-Azpeitia and J.~Ize, \emph{Bifurcation of
periodic solutions from a ring configuration in the vortex and filament
problems,} J. Differential Equations 252 (2012), pp. 5662-5678.

\bibitem {GaIz11}C.~Garc\'{\i}a-Azpeitia and J.~Ize, \emph{Global bifurcation
of polygonal relative equilibria for masses, vortices and dNLS oscillators},
J. of Differential Equations 251 (2011), pp. 3202--3227.

\bibitem {IzVi03}J. Ize and A. Vignoli, \emph{Equivariant Degree Theory,} De
Gruyter Series in Nonlinear Analysis and Applications 8, Walter de Gruyter,
Berlin, 2003.

\bibitem {Jo04}M. Johansson, \emph{Hamiltonian {Hopf} bifurcations in the
discrete nonlinear {Schr{\"{o}}dinger} trimer: oscillatory instabilities,
quasiperiodic solutions and a 'new' type of self-trapping transition}, J.
Phys. A: Math. Gen. 37 (2004), pp. 2201--2222.

\bibitem {Kr}P. Kevrekidis, \emph{The discrete nonlinear {Schr{\"{o}}dinger}
equation, }Mathematical Analysis, Numerical Computations and Physical
Perspectives, Springer, 2009.

\bibitem {MA94}R. MacKay and S. Aubry, \emph{Proof of existence of breathers
for time-reversible or hamiltonian networks of weakly coupled oscillators,
}Nonlinearity 7 (1994), pp. 1623--1643.

\bibitem {MaFu}B. A. Malomed, J. Fujioka, A. Espinosa-Cer\'{o}n, R.
Rodr\'{\i}guez, and S. Gonz\'{a}lez, \emph{Moving embedded lattice solitons,}
Chaos 16 (2006), 013112.

\bibitem {MeHa91}K. Meyer and G. Hall, \emph{An Introduction to Hamiltonian
Dynamical Systems,} Springer-Verlag, 1991.

\bibitem {Ni2001}L. Nirenberg, \emph{Topics in Nonlinear Functional Analysis,}
Courant Lecture Notes in Mathematics 6, American Mathematical Society, 2001.

\bibitem {Pan10}P. Panayotaros, \emph{Continuation of normal modes in finite
{NLS} lattices, }Phys. Lett. A 374 (2010), pp. 3912--1919.

\bibitem {PaPe08}P. Panayotaros and D. Pelinovsky, \emph{Periodic oscillations
of discrete NLS solitons in the presence of diffraction management},
Nonlinearity 21 (2008), pp. 1265-1279.

\bibitem {PD09}C. Pando and E. Doedel, \emph{Bifurcation structures and
dominant models near relative equilibria in the one-dimensional discrete
nonlinear {Schr{\"{o}}dinger} equation}, Physica D. 238 (2009), pp. 687--698.

\bibitem {PeRo}D. Pelinovsky and V. Rothos, \emph{Bifurcations of travelling
wave solutions in the discrete NLS equations,} Physica D. 202 (2005), pp. 16--36.

\bibitem {Ra}P. H. Rabinowitz, \emph{Some global results for nonlinear
eigenvalue problems,} J. Funct. Anal. 7 (1971), pp. 487-513.

\bibitem {Va82}A. Vanderbauwhede, \emph{Local bifurcation and symmetry,}
Research notes in mathematics 75, Pitman Advanced Publishing Program, 1982.
\end{thebibliography}
\end{document}